\documentclass[12pt,reqno]{amsart}
\usepackage{amscd,amsmath,amsthm,amssymb}
\usepackage{color}
\usepackage{pstricks}
\usepackage{stmaryrd}
\usepackage{tikz-cd}
\usepackage{url}

\usepackage{latexsym}
\usepackage{amsfonts,amsmath,mathtools}
\usepackage{graphics}
\usepackage{float}
\usepackage{enumitem}

\newpsstyle{fatline}{linewidth=1.5pt}
\newpsstyle{fyp}{fillstyle=solid,fillcolor=verylight}
\definecolor{verylight}{gray}{0.97}
\definecolor{light}{gray}{0.9}
\definecolor{medium}{gray}{0.85}
\definecolor{dark}{gray}{0.6}

\def\NZQ{\mathbb}               

\def\ZZ{{\NZQ Z}}

\def\frk{\mathfrak}               

\def\mm{{\frk m}}

\def\nn{{\frk n}}
\def\G{{\mathcal G}}

\def\pd{\textup{proj}\phantom{.}\!\textup{dim}}

\def\opn#1#2{\def#1{\operatorname{#2}}} 
\opn\chara{char} \opn\length{\ell} \opn\pd{pd} \opn\rk{rk}
\opn\projdim{proj\,dim} \opn\injdim{inj\,dim} \opn\rank{rank}
\opn\depth{depth} \opn\grade{grade} \opn\height{height}
\opn\embdim{emb\,dim} \opn\codim{codim}

\opn\Tr{Tr} \opn\bigrank{big\,rank}
\opn\superheight{superheight}\opn\lcm{lcm}
\opn\trdeg{tr\,deg}
	\opn\reg{reg} \opn\lreg{lreg} \opn\ini{in} \opn\lpd{lpd}
	\opn\size{size} \opn\sdepth{sdepth}
	\opn\link{link}\opn\fdepth{fdepth}\opn\lex{lex}
	\opn\tr{tr}
	\opn\type{type}
	\opn\gap{gap}
	\opn\diam{diam}
	\opn\Mod{Mod}
	\opn\div{div} \opn\Div{Div} \opn\cl{cl} \opn\Cl{Cl}
	\opn\Spec{Spec} \opn\Supp{Supp} \opn\supp{supp} \opn\Sing{Sing}
	\opn\Ass{Ass} \opn\Min{Min}\opn\Mon{Mon}
	\opn\Ann{Ann} \opn\Rad{Rad} \opn\Soc{Soc}
	\opn\Im{Im} \opn\Ker{Ker} \opn\Coker{Coker} \opn\Am{Am}
	\opn\Hom{Hom} \opn\Tor{Tor} \opn\Ext{Ext} \opn\End{End}
	\opn\Aut{Aut} \opn\id{id}
	
	\opn\nat{nat}
	\opn\pff{pf}
	\opn\Pf{Pf} \opn\GL{GL} \opn\SL{SL} \opn\mod{mod} \opn\ord{ord}
	\opn\Gin{Gin} \opn\Hilb{Hilb}\opn\sort{sort}
	\opn\PF{PF}\opn\Ap{Ap}
	\opn\dist{dist}
	\opn\aff{aff}
	\opn\relint{relint} \opn\st{st}
	\opn\lk{lk} \opn\cn{cn} \opn\core{core} \opn\vol{vol}  \opn\inp{inp} \opn\nilpot{nilpot}
	\opn\link{link} \opn\star{star}\opn\lex{lex}\opn\set{set}
	\opn\width{wd}
	\opn\Fr{F}
	\opn\QF{QF}
	\opn\G{G}
	\opn\type{type}\opn\res{res}
	\opn\conv{conv}
	\opn\sr{sr}
	\opn\gr{gr}
	
	\def\pot#1#2{#1[\kern-0.28ex[#2]\kern-0.28ex]}

	\opn\dirlim{\underrightarrow{\lim}}
	\opn\inivlim{\underleftarrow{\lim}}
	\def\Implies{\ifmmode\Longrightarrow \else
		\unskip${}\Longrightarrow{}$\ignorespaces\fi}
	\def\implies{\ifmmode\Rightarrow \else
		\unskip${}\Rightarrow{}$\ignorespaces\fi}
	\def\iff{\ifmmode\Longleftrightarrow \else
		\unskip${}\Longleftrightarrow{}$\ignorespaces\fi}

	\let\:=\colon
	\newtheorem{Theorem}{Theorem}
	\newtheorem{Theorem1}{Theorem1}
	\newtheorem{Theorem2}{Theorem2}
	\newtheorem{Theorem3}{Theorem3}
	\newtheorem{Theorem4}{Theorem4}
	\newtheorem{Theorem5}{Theorem5}
	\newtheorem{Theorem6}{Theorem6}
	\newtheorem{Theorem7}{Theorem7}
	\newtheorem{Theorem8}{Theorem8}
	\newtheorem{Lemma}[Theorem1]{Lemma}
	\newtheorem{Corollary}[Theorem2]{Corollary}
	\newtheorem{Proposition}[Theorem3]{Proposition}
	\newtheorem{Remark}[Theorem4]{Remark}
	\newtheorem{Examples}[Theorem5]{Examples}
	\newtheorem{Example}[Theorem6]{Example}
	\newtheorem{Definition}[Theorem7]{Definition}
	\newtheorem{Question}[Theorem8]{Question}
	\let\epsilon\varepsilon
	\let\kappa=\varkappa
	\def\qed{\ifhmode\textqed\fi
		\ifmmode\ifinner\hfill\quad\qedsymbol\else\dispqed\fi\fi}
	\def\textqed{\unskip\nobreak\penalty50
		\hskip2em\hbox{}\nobreak\hfill\qedsymbol
		\parfillskip=0pt \finalhyphendemerits=0}
	\def\dispqed{\rlap{\qquad\qedsymbol}}
	
	\opn\dis{dis}
	\def\pnt{{\raise0.5mm\hbox{\large\bf.}}}
	
	\opn\Lex{Lex}
	\opn\Max{Max}
	\opn\Shad{Shad}
	\opn\astab{astab}

	\opn\v{v}

\begin{document}
	
	\title{Cohen-Macaulayness of vertex splittable monomial ideals}
	\author{Marilena Crupi, Antonino Ficarra}
	
	\address{Marilena Crupi, Department of mathematics and computer sciences, physics and earth sciences, University of Messina, Viale Ferdinando Stagno d'Alcontres 31, 98166 Messina, Italy}
	\email{mcrupi@unime.it}
	
	\address{Antonino Ficarra, Department of mathematics and computer sciences, physics and earth sciences, University of Messina, Viale Ferdinando Stagno d'Alcontres 31, 98166 Messina, Italy}
	\email{antficarra@unime.it}
	
	\date{}

 	\begin{abstract}
		In this paper, we give a new criterion for the Cohen--Macaulayness of vertex splittable ideals, a family of monomial ideals recently introduced by Moradi and Khosh-Ahang. Our result relies on a Betti splitting of the ideal and provides an inductive way of checking the Cohen--Macaulay property. As a result, we obtain characterizations for Gorenstein, level and pseudo-Gorenstein vertex splittable ideals. Furthermore, we provide new and simpler combinatorial proofs of known Cohen--Macaulay criteria for several families of monomial ideals, such as (vector-spread) strongly stable ideals and (componentwise) polymatroidals. Finally, we characterize the {\color{black} family} of bi-Cohen--Macaulay graphs by the novel criterion for the Cohen--Macaulayness of vertex splittable ideals.
	\end{abstract}
	
	\maketitle
	
	\section{Introduction}
	Let $S=K[x_1,\dots,x_n]$ be the polynomial ring with coefficients in a field $K$. In \cite{MKA16}, Moradi and Khosh-Ahang introduced the notion of a \emph{vertex splittable ideal}, an algebraic analog of the vertex decomposability property of a simplicial complex. In more detail, let $\Delta$ be a simplicial complex and let $F$ be a face of $\Delta$. One can associate with $\Delta$ two special simplicial complexes: the~deletion of $F$, defined as $\textup{del}_{\Delta}(F) = \{G\in \Delta : F\cap G=\emptyset\}$, and~the link of $F$, defined as $\textup{lk}_{\Delta}(F) = \{G\in \Delta : F\cap G=\emptyset, F\cup G\in \Delta\}$. For~$F=\{x\}$, one sets $\textup{del}_{\Delta}(\{x\})= \textup{del}_{\Delta}(x)$ and $\textup{lk}_{\Delta}(\{x\})=\textup{lk}_{\Delta}(x)$. The~notion of vertex decomposition was introduced  by Provan and Billera~\cite{PB1980} for a pure simplicial complex, and afterwards, it was extended to nonpure complexes by Bjo\"rner and Wachs~\cite{BW1997}. A~\emph{vertex decomposable simplicial complex} $\Delta$ is
	recursively defined as follows: $\Delta$ is a simplex or $\Delta$ has some vertex
	$x$ such that (1) both $\textup{del}_{\Delta}x$ and $\textup{lk}_{\Delta}x$ are vertex decomposable, and~(2) there is no face of $\textup{lk}_{\Delta}x$ which is also a facet of $\textup{del}_{\Delta}x$. An~ideal $I$ of $S$ is called \emph{vertex decomposable} if $I=I_{\Delta}$, with $\Delta$ being a vertex decomposable simplicial complex. We recall that $I_{\Delta}$ is the Stanley--Reisner ideal of $\Delta$ over $K$, that is, the~ideal of $S$ generated by the squarefree monomial $x_F= \prod_{x_j\in F}x_j$, for~all $F\in \Delta$. It is well-known that for a simplicial complex $\Delta$, the following implications hold: \emph{vertex decomposable}$\Rightarrow$ \emph{shellable} $\Rightarrow$ \emph{sequentially Cohen--Macaulay} (see, for~instance,~\cite{RV}). Moreover, there exist characterizations of 
	shellable, sequentially Cohen--Macaulay and Cohen--Macaulay complexes $\Delta$ via the Alexander dual ideals $I_{\Delta^\vee}$ (see \cite[Theorem 1.4]{HTZ2004}, \cite[Theorem 2.1]{HH1999}, \cite[Theorem 3]{ER1998}, respectively), where $\Delta^\vee = \{X\setminus F: F\notin \Delta\}$ is the Alexander dual simplicial complex associated with $\Delta$.  
	
	Inspired by the above results, in~\cite{MKA16}, Moradi and Khosh-Ahang asked and solved the following question: \emph{Is it possible to characterize a vertex decomposable simplicial complex $\Delta$ by means of $I_{\Delta^\vee}$?} For this aim, they introduced the notion of the vertex splittable monomial ideal (Definition \ref{def:vertex splittable}) and proved that a simplicial complex $\Delta$ is vertex decomposable if and only if $I_{\Delta^\vee}$ is a vertex splittable ideal \cite[Theorem 2.3]{MKA16}. Moreover, the~authors in~\cite{MKA16} {\color{black} proved} that a vertex splittable ideal has a Betti splitting ({\color{black} see Definition \ref {def:Bettisplitting} and Theorem \ref{thm:MKA16})}. 
	
	{\color{black} Determining when a monomial ideal is Cohen--Macaulay is a fundamental and challenging problem in commutative algebra. Motivated by this and the results of~\cite{MKA16}, in~this paper, we tackle the Cohen--Macaulayness of vertex splittable ideals. Our main contribution (Theorem \ref{Thm:CM-VS}) is a new characterization of the Cohen--Macaulayness of a vertex splittable ideal in terms of a Betti splitting. This new criterion provides a neat and effective inductive strategy to determine when a vertex splittable ideal is Cohen--Macaulay.}
	
	This article is organized as follows. In~Section~\ref{secVS:1}, we recall relevant definitions and auxiliary
	results that we will {\color{black} use later on.}  In Section~\ref{secVS:2}, we state a new criterion for the Cohen--Macaulayness of vertex splittable ideals (Theorem \ref{Thm:CM-VS}). 
	As a consequence, we obtain characterizations for Gorenstein, level and pseudo-Gorenstein Cohen--Macaulay
	vertex splittable ideals (Corollary \ref{Cor:CM-VS}). The~results in this section will be used in the subsequent section (Section \ref{secVS:3}), where we recover some interesting Cohen--Macaulay classifications of families of monomial ideals: (vector-spread) strongly stable ideals and~(componentwise) polymatroidal
	ideals. Moreover,  a~new characterization of bi-Cohen--Macaulay graphs is presented (Theorem \ref{Thm:BCM-VS}). Finally, Section~\ref{sec:5}  contains our conclusions and perspectives.
	
	\section{Preliminaries}
	In this section, we recall the basic notions and notations we will use in the body of the paper~\cite{FHT2009, MKA16}.
	
	Let $S=K[x_1,\dots,x_n]$ be a polynomial ring in $n$ variables over a field $K$ with the standard grading, i.e.,~each $\deg x_i =1$.  For~any finitely generated graded $S$-module $M$, there exists the unique minimal graded free $S$-resolution
	\begin{equation*}\label{eq:free}
		F: 0\rightarrow F_p\xrightarrow{\ d_p\ } F_{p-1}\xrightarrow{d_{p-1}}\cdots\xrightarrow{\ d_2\ }F_1\xrightarrow{\ d_1\ }F_0\xrightarrow{\ d_0\ }M\rightarrow0,
	\end{equation*}
	with $F_i=\bigoplus_jS(-j)^{\beta_{i,j}}$. The~numbers $\beta_{i,j}= \beta_{i,j}(M)$ are called the graded Betti numbers of $M$, while $\beta_i(M) = \sum_j \beta_{i,j}(M)$ are called the total Betti numbers of $M$. Recall that the \emph{projective dimension} and the \emph{Castelnuovo--Mumford regularity} of $M$ are defined as follows:
	\begin{align*}
		\pd M\ & =\ \max\{i\ :\ \beta_i(M)\neq 0\},\\
		\reg M\ &=\ \max\{\mbox{$j-i$\ :\ $\beta_{i,j}(M)\ne0$, for~some  $i$ and $j$}\}.
	\end{align*}
	
	More precisely, the~projective dimension $\pd(M)$ is the length of a minimal graded free resolution of the finitely generated graded $S$-module $M$. 
	
	\subsection{Cohen-Macaulay~Property}\label{sub1}
	
	In this subsection, we introduce the notion of the Cohen--Macaulay ring and some related~notions.
	
	Firstly, let $\mm=(x_1,\dots,x_n)$ be the unique maximal homogeneous ideal of $S$, and let $M$ be a finitely generated graded $S$-module $M$.
	
	A sequence ${\bf f} = f_1,\ldots, f_d$ of homogeneous elements of $\mm$ is
	called an \emph{$M$-sequence}  if the following criteria are met:
	\begin{enumerate}
		\item[(a)]  the multiplication map $M/(f_1,\ldots, f_{i-1})M\xrightarrow{f_i}M/(f_1,\ldots, f_i)M$ is injective for all $i$. 
		\item[(b)] $M/({\bf f} )M\neq 0$. 
	\end{enumerate}
	
	The length of a maximal homogeneous $M$-sequence is called the \emph{depth} of $M$. By~the Auslander--Buchsbaum formula (see, for~instance,~\cite{BH}), we have
	\begin{equation}\label{eq:ABformula}
		\depth M + \pd M = n.
	\end{equation}
	
	A finitely generated graded $S$-module $M$ is called a \emph{Cohen--Macaulay module} (CM module for~short) if $\depth M = \dim M$, where $\dim M$ is the Krull dimension of $M$ \cite{BH}. Let $I$ be a graded ideal of $S$; the~graded ring $S/I$ is said to be CM if $S/I$, viewed as an $S$-module, is CM. The~graded ideal $I$ is called a CM~ideal.

	Let $I\subset S$ be a graded CM ideal, and~let $p=\pd(S/I)$ be the projective dimension of $S/I$. The~\textit{Cohen--Macaulay type} (CM type for~short) of $S/I$ is defined as the integer $\textup{CM-type}(S/I)=\beta_p(S/I)$. It is well-known that $S/I$ is Gorenstein if and only if $\textup{CM-type}(S/I)=1$. We say that $I$ is Gorenstein if $S/I$ is~such.
	
	By \cite[Corollary 2.17]{HHOBook}, the graded Betti number $\beta_{p,p+\reg S/I}(S/I)$ is always nonzero. 
	We say that $S/I$ is \textit{level} if and only if $\beta_p(S/I)=\beta_{p,p+\reg S/I}(S/I)$. Following~\cite{EHHM15}, we say that $S/I$ is \textit{pseudo-Gorenstein} if and only if $\beta_{p,p+\reg S/I}(S/I)=1$. Hence, $S/I$ is Gorenstein if and only if it is both level and~pseudo-Gorenstein.
	
	For more details on this subject, see, for instance,~\cite{BH, JT, EHGB,HHOBook}.

	\subsection{Vertex Splittable Monomial~Ideals}\label{secVS:1}
	
	In this subsection, we discuss the notions of vertex splittable monomial ideals and of Betti~splittings.
	
	Let $I\subset S$ be a monomial ideal.  
	We denote by $\mathcal{G}(I)$ the unique minimal monomial generating set of $I$. We recall the following notion \cite[Definition 2.1]{MKA16}.
	\begin{Definition}\label{def:vertex splittable}
		\rm The ideal $I$ is called \textit{vertex splittable} if it can be obtained by the following recursive~procedure:
		\begin{enumerate}
			\item[\textup{(i)}] If $u$ is a monomial and $I=(u)$, $I=0$ or $I=S$, then $I$ is vertex splittable.\smallskip
			\item[\textup{(ii)}] If there exists a variable $x_i$ and vertex splittable ideals $I_1\subset S$ and $I_2\subset K[x_1,\dots,x_{i-1},$ \newline $ x_{i+1}, \dots,x_n]$ such that $I=x_iI_1+I_2$, $I_2\subseteq I_1$ and $\mathcal{G}(I)$ is the union of $\mathcal{G}(x_iI_1)$ and $\mathcal{G}(I_2)$, then $I$ is vertex splittable. \\ In this case, we say that $I=x_iI_1+I_2$ is a \textit{vertex splitting} of $I$ and $x_i$ is a \textit{splitting vertex} of $I$.
		\end{enumerate}
	\end{Definition}
	
	\begin{Remark}\rm {\color{black} One can observe that while in general, the~Cohen--Macaulayness of $S/I$ depends on the field $K$ (\cite[page 236]{BH}, if~$I$ is a vertex splittable ideal, then this is not the case.  
			Indeed, the~Krull dimension of $S/I$, where $I$ is a monomial ideal, does not depend on $K$. Furthermore, by~\cite[Theorem 2.4]{MKA16}, vertex splittable ideals have linear quotients. Hence, $\depth S/I$ is also independent from $K$. }
	\end{Remark}
	
	In~\cite{FHT2009}, the~next concept was~introduced.
	
	\begin{Definition}\label{def:Bettisplitting}
		\rm Let $I$, $J$, $L$ be monomial ideals of $S$ such that $\mathcal{G}(I)$ is the disjoint union of $\mathcal{G}(J)$ and $\mathcal{G}(L)$. We say that $I=J+L$ is a \textit{Betti splitting} if
		\begin{equation}\label{eq:BettiSplittingI=P+Q}
			\beta_{i,j}(I)=\beta_{i,j}(J)+\beta_{i,j}(L)+\beta_{i-1,j}(J\cap L), \ \ \ \textup{for all}\ i,j.
		\end{equation}
	\end{Definition}
	
	When $I = J+L$ is a Betti splitting, important homological invariants of the ideal $I$ are related to the invariants of the smaller ideals $J$ and $L$. Indeed, in~
	\cite[Corollary 2.2]{FHT2009}, it is proved that if $I=J+L$ is a Betti splitting, then
	\begin{equation}\label{eq:pdBettiSplit}
		\pd I = \max\{\pd J, \pd L, \pd (J\cap L)+1\}.
	\end{equation}
	
	We quote the next crucial result from~\cite{MKA16}.
	\begin{Theorem}\label{thm:MKA16}{\rm \cite[Theorem 2.8]{MKA16}} Let $I = xI_1+ I_2$  be a vertex splitting for the monomial ideal $I$ of $S$. Then $I = xI_1+ I_2$ is a Betti splitting.
	\end{Theorem}
	
	For recent applications of vertex splittings, see the papers~\cite{DRSVT23,M2023}.
	
	{\color{black} We close this subsection by introducing two families of monomial ideals: the $t$-spread strongly stable ideals and the (componentwise) polymatroidal ideals. We will show in \mbox{Section~\ref{secVS:3}} that they are families of vertex splittable ideals (see Propositions~\ref{prop:vec} and} \ref{Prop:CPVS}).
	
	A very meaningful class of monomial ideals of the polynomial ring $S$ is the class of \emph{strongly stable} monomial ideals. They are fundamental in commutative algebra, because~if $K$ has the characteristic zero, then they appear as generic initial ideals~\cite{HH2011}. In~\cite{F1}, the~concept of a strongly stable ideal was generalized to that of the ${\bf t}$-\emph{spread strongly stable} ideal.
	
	Let $d\ge2$, ${\bf t}=(t_1,\dots,t_{d-1})\in\ZZ^{d-1}_{\ge0}$ be a $(d-1)$~tuple, and~let $u=x_{i_1}\cdots x_{i_\ell}\in S$ be a monomial, with~$1\le i_1\le\dots\le i_\ell\le n$ and $\ell\le d$. We say that $u$ is ${\bf t}$\textit{-spread} if
	$$
	i_{j+1}-i_j\ge t_j\ \ \textup{for all}\ \ j=1,\dots,\ell-1.
	$$
	
	A monomial ideal $I\subset S$ is called ${\bf t}$\textit{-spread} if $\mathcal{G}(I)$ consists of ${\bf t}$-spread monomials. A~${\bf t}$-spread ideal $I\subset S$ is called \textit{${\bf t}$-spread strongly stable} if for all ${\bf t}$-spread monomials $u\in I$ and all $i<j$ such that $x_j$ divides $u$ and $x_i(u/x_j)$ is ${\bf t}$-spread, then $x_i(u/x_j)\in I$. \linebreak  For~${\bf t}=(0,\dots,0)$ and ${\bf t}=(1,\dots,1)$, we obtain the strongly stable and the squarefree strongly stable ideals~\cite{JT}.

	Another fundamental family of monomial ideals of $S$ is that of the so-called \emph{polymatroidal ideals}.
	
	Let $I\subset S$ be a monomial ideal generated in a single degree. We say that $I$ is \textit{polymatroidal} if the set of the exponent vectors of the minimal monomial generators of $I$ is the set of bases of a discrete polymatroid~\cite{JT}. 
	
	Polymatroidal ideals are characterized by the \textit{exchange property} \cite[Theorem 2.3]{JT}. For~a monomial $u\in S$, let
	$$
	\deg_{x_i}(u)\ =\ \max\{j\ :\ x_i^j\ \textup{divides}\ u\}.
	$$
	\vspace{-17pt}
	{\color{black} \begin{Lemma}  \label{lem:exchange} Let $I\subset S$ be a monomial ideal generated in a single degree. Then $I$ is polymatroidal if and only if 
			the following exchange property holds: for all $u,v\in\mathcal{G}(I)$ and all $i$ such that $\deg_{x_i}(u)>\deg_{x_i}(v)$, there exists $j$ with $\deg_{x_j}(u)<\deg_{x_j}(v)$ such that $x_j(u/x_i)\in\mathcal{G}(I)$.
	\end{Lemma}}

	An arbitrary monomial ideal $I$ is called \textit{componentwise polymatroidal} if the component $I_{\langle j\rangle}$ is polymatroidal for all $j$. Here, for~a graded ideal $J\subset S$ and an integer $j$, we denote by $J_{\langle j\rangle}$ the graded ideal generated by all polynomials of degree $j$ belonging to $J$.
	
	Polymatroidal ideals are vertex splittable \cite[Lemma 2.1]{MNS22}. In~Proposition \ref{Prop:CPVS}, we prove the analogous case for componentwise polymatroidal~ideals.

	\section{A Cohen--Macaulay~Criterion}\label{secVS:2}
	In this section, we introduce a new criterion for the Cohen--Macaulayness of vertex splittable ideals. As~a {\color{black} result, we obtain} characterizations for Gorenstein, level and pseudo-Gorenstein vertex splittable~ideals. 
	
	The main result in the section is the~following.
	
	\begin{Theorem}\label{Thm:CM-VS}
		Let $I\subset S$ be a vertex splittable monomial ideal such that $I\subseteq\mm^2$, and let $x_i$ be a splitting vertex of $I$. Then, the~following conditions are~equivalent:
		\begin{enumerate}
			\item[\textup{(a)}] $I$ is CM.
			\item[\textup{(b)}] $(I:x_i),(I,x_i)$ are CM, and $\depth S/(I:x_i)=\depth S/(I,x_i)$.
		\end{enumerate}
	\end{Theorem}
	\begin{proof}
		We may assume $i=1$. Let $I=x_1I_1+I_2$ be the vertex splitting of $I$. Since \linebreak  $I=x_iI_1+I_2$ is a Betti splitting (Theorem \ref{thm:MKA16}),  
		then Formula (\ref{eq:pdBettiSplit}) 
		together with the Auslander--Buchsbaum Formula (\ref{eq:ABformula}), implies
		$$
		\depth S/I\ =\ \min\{\depth S/(x_1I_1),\,\depth S/I_2,\,\depth S/(x_1I_1\cap I_2)-1\}.
		$$
		
		Notice that $\depth S/x_1I_1=\depth S/I_1$ and $x_1I_1\cap I_2=x_1(I_1\cap I_2)=x_1I_2$, because~$I_2\subseteq I_1$ and $x_1$ does not divide any minimal monomial generator of $I_2$. Consequently, \linebreak  $\depth S/(x_1I_1\cap I_2)=\depth S/(x_1I_2)=\depth S/I_2$, and~so
		\begin{equation}\label{eq:depthBettiSplitVS}
			\depth S/I\ =\ \min\{\depth S/I_1,\,\depth S/I_2-1\}.
		\end{equation}
		
		We have the short exact sequence
		$$
		0\rightarrow S/(I:x_1)\rightarrow S/I\rightarrow S/(I,x_1)\rightarrow0.
		$$
		
		Notice that $(I:x_1)=(x_1I_1+I_2):x_1=(x_1I_1:x_1)+(I_2:x_1)=I_1+I_2=I_1$, because~$x_1$ does not divide any minimal monomial generator of $I_2$ and $I_2\subseteq I_1$. Since $I\subseteq\mm^2$, we have $x_1\notin I$. Thus, $I_1\ne S$.  Moreover, $(I,x_1)=(x_1I_1+I_2,x_1)=(I_2,x_1)$, and so 
		we obtain the short exact sequence
		$$
		0\rightarrow S/I_1\rightarrow S/I\rightarrow S/(I_2,x_1)\rightarrow0.
		$$
		Hence, $\dim S/I=\max\{\dim S/I_1,\dim S/(I_2,x_1)\}$. Since $S/(I_2,x_1)\cong K[x_2,\dots,x_n]/I_2$, we obtain that $\dim S/(I_2,x_1)=\dim K[x_2,\dots,x_n]/I_2=\dim S/I_2-1$. Hence,
		\begin{equation}\label{eq:dimBettiSplitVS}
			\dim S/I\ =\ \max\{\dim S/I_1,\,\dim S/I_2-1\}.
		\end{equation}
		
		(a)$\Rightarrow$(b) Suppose that $I$ is CM. By~Equations~(\ref{eq:depthBettiSplitVS}) and (\ref{eq:dimBettiSplitVS}), we have
		$$
		\dim S/I\ \ge\ \dim S/I_1\ \ge\ \depth S/I_1\ \ge\ \depth S/I\ =\ \dim S/I,
		$$
		and
		$$
		\dim S/I\ \ge\ \dim S/I_2-1\ \ge\ \depth S/I_2-1\ \ge\ \depth S/I\ =\ \dim S/I.
		$$
		
		Hence, $S/I_1$, $S/I_2$ are CM and $\depth S/I_1=\depth S/I=\depth S/I_2-1$.\medskip
		
		(b)$\Rightarrow$(a) Conversely, assume that $S/I_1$ and $S/I_2$ are CM and that \linebreak  $\depth S/I_1=\depth S/I_2-1$. Then,
		$$
		\min\{\depth S/I_1,\,\depth S/I_2-1\}\ =\ \depth S/I_1,
		$$
		and
		$$
		\max\{\dim S/I_1,\,\dim S/I_2-1\}\ =\ \dim S/I_1.
		$$
		
		Equations~(\ref{eq:depthBettiSplitVS}) and (\ref{eq:dimBettiSplitVS}) imply that $\depth S/I=\depth S/I_1=\dim S/I_1=\dim S/I$ and so $S/I$ is CM.
	\end{proof}

	The next important vanishing theorem due to Grothendieck \cite[Theorem 3.5.7]{BH} will be crucial to characterize Gorenstein, level and pseudo-Gorenstein vertex splittable ideals. If~$(R, \mm, k)$ is a Noetherian local ring and $M$ a finitely generated $R$-module, we denote by $H^i_\mm(M)$ the $i$th local cohomology module of $M$ with support on $\mm$ \cite{BH}.

	\begin{Theorem} \label{thm:Gro} Let $(R, \mm, k)$ be a Noetherian local ring and $M$ a finitely generated $R$-module of depth $t$ and dimension $d$.~Then
		\begin{enumerate}
			\item[\textup{(a)}] $H^i_\mm(M)=0$ 
			for~$i<t$ and $i>d$.
			\item[\textup{(b)}] $H^t_\mm(M)\neq 0$ and $H^d_\mm(M)\neq 0$.
		\end{enumerate}  
	\end{Theorem}
	
	\begin{Corollary}\label{Cor:CM-VS}
		Let $I\subset S$ be a vertex splittable CM ideal such that $I\subseteq\mm^2$, and let $I=x_iI_1+I_2$ be a vertex splitting of $I$. Then, the~following statements~hold:
		\begin{enumerate}
			\item[\textup{(a)}] $\textup{CM-type}(S/I)=\textup{CM-type}(S/I_1)+\textup{CM-type}(S/I_2)$.
			\item[\textup{(b)}] $S/I$ is Gorenstein if and only if $I$ is a principal ideal.
			\item[\textup{(c)}] $S/I$ is level if and only if $S/I_1$ and $S/I_2$ are level and $\reg S/I_1+1=\reg S/I_2$.
			\item[\textup{(d)}] $S/I$ is pseudo-Gorenstein if and only~if one of the following occurs: 
			Either $S/I_1$ is pseudo-Gorenstein and $\reg S/I_1+1>\reg S/I_2$ or $S/I_2$ is pseudo-Gorenstein and $\reg S/I_1+1<\reg S/I_2$.
			
			\item[\textup{(e)}] $H^{\dim S/I}_\mm(S/I)/H^{\dim S/I}_\mm(S/(I:x_i))\cong H^{\dim S/I}_\mm(S/(I,x_i))$.
		\end{enumerate}
		
	\end{Corollary}
	\begin{proof}
		We may assume that $x_i=x_1$. Since $S/I$ is CM, Theorem \ref{Thm:CM-VS} guarantees that $S/I_1,S/I_2$ are CM and $\depth S/I_1=\depth S/I_2-1$. Hence, $\pd S/I_1=\pd S/I_2+1$. Let $p=\pd S/I_1$. In~particular, we have $p=\pd S/I$. Now, by~\cite[Remark 2.10]{MKA16}, we have for all $j$
		$$
		\beta_{p,p+j}(S/I)\ =\ \beta_{p,p+j-1}(S/I_1)+\beta_{p,p+j}(S/I_2)+\beta_{p-1,p-1+j}(S/I_2).
		$$
		
		Since $\pd S/I_2=p-1$, 
		the above formula simplifies to
		\begin{equation}\label{eq:BettiVS}
			\beta_{p,p+j}(S/I)\ =\ \beta_{p,p+j-1}(S/I_1)+\beta_{p-1,p-1+j}(S/I_2).
		\end{equation}
		
		From this formula, we deduce that
		$$
		\reg S/I\ =\ \max\{\reg S/I_1+1\,\ \reg S/I_2\}.
		$$
		
		We obtain the following:  
		
		(a) The assertion follows immediately from (\ref{eq:BettiVS}).
		
		(b) In the proof of Theorem \ref{Thm:CM-VS} we noted, that $I_1\ne 0,S$. Thus, $\textup{CM-type}(S/I_1)\ge1$. By~(a), it follows that $I$ is Gorenstein if and only if 
		$I_1$ is Gorenstein, and~        $I_2=0$. 
		Using Formula~(\ref{eq:depthBettiSplitVS}) and Theorem~\ref{Thm:CM-VS} (b),
		we obtain $\depth S/I=n-1$. Since $\depth S/I=n - \pd S/I$, we have $\pd S/I = 1$, equivalent to saying that  $I$ is a principal ideal.   
		
		(c) Assume that $S/I$ is level. Then $\beta_{p,p+j}(S/I)\ne0$ only for $j=\reg S/I$. Since $\reg S/I=\max\{\reg S/I_1+1,\reg S/I_2\}$ and $\beta_{p,p+\reg S/I_1}(S/I_1),\beta_{p-1,p-1+\reg S/I_2}(S/I_2)$ are both nonzero, we deduce from Formula~(\ref{eq:BettiVS}) that $\reg S/I_1+1=\reg S/I_2$ and that $S/I_1$, $S/I_2$ are level. Conversely, if~$\reg S/I_1+1=\reg S/I_2$ and $S/I_1$, $S/I_2$ are level, we deduce from Formula (\ref{eq:BettiVS}) that $S/I$ is level.
		
		(d) Assume that $S/I$ is pseudo-Gorenstein. Then $\beta_{p,p+\reg S/I}(S/I)=1$. Since $\reg S/I=\max\{\reg S/I_1+1,\reg S/I_2\}$ and $\beta_{p,p+\reg S/I_1}(S/I_1),\beta_{p-1,p-1+\reg S/I_2}(S/I_2)$ are both nonzero, we deduce from Formula (\ref{eq:BettiVS}) that either $S/I_1$ is pseudo-Gorenstein and $\reg S/I_1+1>\reg S/I_2$ or~$S/I_2$ is pseudo-Gorenstein and $\reg S/I_1+1<\reg S/I_2$. The~converse can be proved in a similar way.
		
		(e) Since $S/I$ is CM, Theorem \ref{Thm:CM-VS} implies that $S/(I:x_i)$ and $S/(I,x_i)$ are CM and $\dim S/(I:x_i)=\dim S/(I,x_i)=\dim S/I$. As~shown in the proof of Theorem \ref{Thm:CM-VS}, we have the short exact sequence
		$$
		0\rightarrow S/(I:x_i)\rightarrow S/I\rightarrow S/(I,x_i)\rightarrow0.
		$$
		
		This sequence induces the long exact sequence of local cohomology modules:
		$$
		\cdots\rightarrow H^{i-1}_\mm(S/(I,x_i))\rightarrow H^{i}_\mm(S/(I:x_i))\rightarrow H^{i}_\mm(S/I)\rightarrow H^{i}_\mm(S/(I,x_i))\rightarrow\cdots.
		$$
		
		Let $M$ be a finitely generated CM $S$-module. By~Theorem \ref{thm:Gro},  
		$H_\mm^i(M)\ne0$ if and only if $i=\depth M=\dim M$. Thus, the~above exact sequence simplifies to
		$$
		0\rightarrow H^{\dim S/I}_\mm(S/(I:x_i))\rightarrow H^{\dim S/I}_\mm(S/I)\rightarrow H^{\dim S/I}_\mm(S/(I,x_i))\rightarrow0,
		$$
		and the assertion follows.
	\end{proof}

	\begin{Remark}\rm It is clear that any ideal $I\subset S$ generated by a subset of the variables of $S$ is Gorenstein and vertex splittable. Hence, Corollary \ref{Cor:CM-VS} implies immediately that the only Gorenstein vertex splittable ideals of $S$ are the principal monomial ideals and the ideals generated by a subset of the variables.
	\end{Remark}

	\section{Families of Cohen--Macaulay Vertex Splittable~Ideals}\label{secVS:3}
	
	In this section, by~using Theorem \ref{Thm:CM-VS}, we recover in a simple and very effective manner {\color{black} Cohen--Macaulay criteria for several} families of monomial ideals. We use the fact that if $I=x_iI_1+I_2$ is a vertex splitting, then $I_1,I_2$ are vertex splittable ideals that, in~good cases, belong again to a given family of vertex splittable monomial ideals and~to which one may apply inductive~arguments.
	
	{\color{black} The first two families were introduced in Section~\ref{secVS:1}.}
	\subsection{(Vector-Spread) Strongly Stable Ideals} 
	
	In \cite[Theorem 4.3]{CF}, we classified the CM ${\bf t}$-spread strongly stable ideals. Here, we recover this result using Theorem \ref{Thm:CM-VS}.
	\begin{Proposition}\label{prop:vec}
		Let $I\subset S$ be a ${\bf t}$-spread strongly stable ideal such that $I\subseteq\mm^2$. Then {\color{black}
			\begin{enumerate}
				\item[\em(a)] $I$ is vertex splittable;
				\item[\em(b)] $I$ is CM if and only if there exists $\ell\le d$ such that
				$$
				x_{n-(t_1+t_2\dots+t_{\ell-1})}x_{n-(t_2+t_3\dots+t_{\ell-1})}\cdots x_{n-t_{\ell-1}}x_n\in\mathcal{G}(I).
				$$
		\end{enumerate}}
	\end{Proposition}
	\begin{proof}
		
		(a) We proceed by double induction on the number of variables $n$ and the highest degree $d$ of a generator $u\in\mathcal{G}(I)$. If~$n=1$, then $I$ is a principal ideal whether or not the integer $d$ is, and~so it is vertex splittable. Suppose $n>1$. If~$d=1$, then $I$ is an ideal generated by a subset of the variables and it is clearly vertex splittable. Suppose $d>1$. We can write $I=x_1I_1+I_2$, where $\mathcal{G}(I_1)=\{u/x_1:u\in\mathcal{G}(I),\ x_1\ \textup{divides}\ u\}$ and $\mathcal{G}(I_2)=\mathcal{G}(I)\setminus\mathcal{G}(x_1I_1)$. It is immediately clear that $I_1\subset S$ is $(t_2,\dots,t_{d-1})$-spread strongly stable and that $I_2$ is a ${\bf t}$-spread strongly stable ideal of $K[x_2,\dots,x_n]$. By~induction on $n$ and $d$,  we have that $I_1$ and $I_2$ are vertex splittable. Hence, so is $I$.
		
		(b) We may suppose that $x_n$ divides some minimal generator of $I$. Otherwise, we can consider $I$ as a monomial ideal of a smaller polynomial ring. If~$I$ is principal, then we have $I=(u)=(x_1x_{1+t_1}\cdots x_{1+t_1+\dots+t_{\ell-1}})$, with~$n=1+t_1+\dots+t_{\ell-1}$, and~$\ell\le d$. Otherwise, if~$I$ is not principal, then $\pd S/I>1$, and we can write $I=x_1I_1+I_2$ as above. By~Theorem~\ref{Thm:CM-VS}, $I_2$ is CM and $\pd S/I_2=\pd S/I+1>1$. Thus, $I_2\ne0$. Hence, by~induction, there exists $\ell\le d$ such that $$x_{n-(t_1+\dots+t_{\ell-1})}\cdots x_{n-t_{\ell-1}}x_n\in\mathcal{G}(I_2).$$ 
		
		Since $\mathcal{G}(I_2)\subset\mathcal{G}(I)$, the~assertion follows.
	\end{proof}
	\subsection{Componentwise Polymatroidal~Ideals}\label{sub4.2}
	
	In this subsection, we prove that componentwise polymatroidal ideals are also vertex~splittable. 
	
	A longstanding conjecture of Bandari and Herzog predicted that componentwise polymatroidal ideals have linear quotients~\cite{BH2013}. This conjecture was solved recently in~\cite[Theorem 3.1]{F2}. Inspecting the proof of this  theorem, we obtain the following:
	
	\begin{Proposition}\label{Prop:CPVS}
		Componentwise polymatroidal ideals are vertex splittable.
	\end{Proposition}
	\begin{proof} Let $I\subset S$ be a componentwise polymatroidal ideal. We prove the statement by induction on $|\mathcal{G}(I)|$.
		We may assume that all variables $x_i$ divide some minimal monomial generator of $I$. Moreover, it holds that 
		for any variable $x_i$ which divides a minimal monomial generator of minimal degree of $I$, we can write $I=x_iI_1+I_2$, where $\mathcal{G}(x_iI_1)=\{u\in \mathcal{G}(I):x_i\ \textup{divides}\ u\}$, $\mathcal{G}(I_2)=\mathcal{G}(I)\setminus\mathcal{G}(x_iI_1)$ and the following properties are satisfied (see the proof of \cite[Theorem 3.1]{F2}):
		\begin{enumerate}
			\item[(i)] $I_2\subseteq I_1$ as monomial ideals of $S$.
			\item[(ii)] $x_iI_1$ is a componentwise polymatroidal ideal of $S$.
			\item[(iii)] $I_2$ is a componentwise polymatroidal ideal of $K[x_1,\dots,x_{i-1},x_{i+1},\dots,x_n]$.
		\end{enumerate}
		By induction, both $I_1$ and $I_2$ are vertex splittable. Hence, so is $I$.
	\end{proof}

	We have the following~corollary.
	
	\begin{Corollary}\label{Cor:ICM-CPI-VS}
		Let $I\subset S$ be a componentwise polymatroidal ideal and let $x_i$ be any variable dividing some minimal monomial generator of least degree of $I$. Suppose that $I\subseteq\mm^2$. Then, the~following conditions are~equivalent.
		\begin{enumerate}
			\item[\textup{(a)}] $I$ is CM.
			\item[\textup{(b)}] $(I:x_i),(I,x_i)$ are CM componentwise polymatroidal ideals and $\depth S/(I:x_i)=\depth S/(I,x_i)$.
		\end{enumerate}
		Moreover, if~$I\subset S$ is a polymatroidal ideal generated in degree $d\ge2$ and $x_i$ is a variable dividing some monomial of $\mathcal{G}(I)$, then $(I:x_i)$ is polymatroidal. And, in~addition, if~$I$ is CM, then $(I:x_i)$ is also CM.
	\end{Corollary}
	\begin{proof} It follows by combining the vertex splitting presented in the proof of Proposition \label{Prop:CPVS} with the facts (ii) and (iii) and with Theorem \ref{Thm:CM-VS}. 
		For the last statement, see \cite[Lemma~5.6]{F3}.
	\end{proof}
	
	At the moment, to~classify all CM componentwise polymatroidal ideals seems {\color{black} a} hopeless task. For~instance, let $J$ be any componentwise polymatroidal ideal. Let $\ell$ be the highest degree of a minimal monomial generator of $J$, and~let $d>\ell$ be any integer. It is easy to see that $I=J+\mm^d$ is componentwise polymatroidal. Since $\dim S/I=0$, then $S/I$ is automatically~CM.
	
	\begin{Example}
		Consider the ideal $I=(x_1^2,x_1x_3,x_3^2,x_1x_2x_4,x_2x_3x_4,x_2^2x_4^2)$ of $S=K[x_1,x_2,x_3,x_4]$, see \cite[Example 3.2]{F2}. One can easily check that $I$ is a CM componentwise polymatroidal ideal. Indeed, it is not difficult to check that $I_{\langle j\rangle}$ is polymatroidal for $j=0,1,2,3,4$. For~$j\ge5$, the~ideal $I_{\langle j\rangle}=\mm^{j-4}I_{\langle 4\rangle}$ is polymatroidal because it is the product of two polymatroidal ideals \cite[Theorem~12.6.3]{HH2011}. Notice that in this case, $\dim S/I>0$.
	\end{Example}
	
	Nonetheless, if~$I$ is generated in a single degree, that is, if~$I$ is actually polymatroidal, then Herzog and Hibi \cite[Theorem 4.2]{HH2003} showed that $I$ is CM if and only if (i) $I$ is a principal ideal, (ii) $I$ is a \textit{squarefree Veronese ideal} $I_{n,d}$, that is, it is generated by all squarefree monomials of $S$ of a given degree $d\le n$, or~(iii) $I$ is a \textit{Veronese ideal}, that is, $I=\mm^d$ for some integer $d\ge1$.
	
	The proof presented by Herzog and Hibi is based on the computation of $\sqrt{I}$. We now present a different proof based on the criterion for Cohen--Macaulayness proved in \mbox{Theorem~\ref{Thm:CM-VS}}.
	\begin{Corollary}\label{Cor:HH-VS}
		A polymatroidal $I\subset S$ is CM if and only if $I$ is one of the following:
		\begin{enumerate}
			\item[\textup{(a)}] A principal ideal.
			\item[\textup{(b)}] A Veronese ideal.~    		
			\item[\textup{(c)}] A squarefree Veronese ideal.
		\end{enumerate}
	\end{Corollary}
	
	For the proof, we need the following well-known identities. For~the convenience of the reader, we provide a proof that uses the vertex splittings technique.
	\begin{Lemma}\label{Lemma:depth-m-Ind-VS}
		Let $n,d\ge1$ be positive integers. Then
		$$
		\pd S/\mm^d=n,\ \ \text{and}\ \ \pd S/I_{n,d}\ =\ n+1-d.
		$$
	\end{Lemma}
	\begin{proof}
		Since $\dim S/\mm^d=0$, we have $\depth S/\mm^d=0$ and $\pd S/\mm^d=n$. If~$d=1$, then $I_{n,1}=\mm$ and $\pd S/I_{n,1}=n$. If~$1<d\le n$, we notice that $I_{n,d}=x_nI_{n-1,d-1}+I_{n-1,d}$ is a vertex splitting. By~Formula (\ref{eq:depthBettiSplitVS}) and induction on $n$ and $d$,
		\begin{align*}
			\pd S/I_{n,d}\ &=\ \min\{\pd S/I_{n-1,d-1},\pd S/I_{n-1,d}+1\}\\
			&=\ \min\{n-1+1-(d-1),n-1+1-d+1\}\\
			&=\ n+d-1,
		\end{align*}
		as wanted.
	\end{proof}
	
	We are now ready for the proof of Corollary \ref{Cor:HH-VS}.
	\begin{proof}[Proof of Corollary \ref{Cor:HH-VS}] 
		Let $I$ be a polymatroidal ideal. We proceed by induction on $|\mathcal{G}(I)|$. If~$|\mathcal{G}(I)|=1$, then $I$ is principal and it is CM. Now, let $|\mathcal{G}(I)|>1$. If~$I=\mm$, then $I$ is CM. Thus, we assume that $I$ is generated in degree $d\ge2$, that all variables $x_i$ divide some minimal monomial generator of $I$ and~that the greatest common divisor of the minimal monomial generators of $I$ is one. By~Proposition \ref{Prop:CPVS} and Corollary \ref{Cor:ICM-CPI-VS}, we have a vertex splitting $I=x_iI_1+I_2$ for each variable $x_i$~$I_1$ and $I_2$ are CM polymatroidal ideals with $\depth S/I_1=\depth S/I_2-1$. Thus, $\pd S/I_1=\pd S/I_2+1$. We may assume that $x_i=x_n$.
		
		Since $|\mathcal{G}(x_nI_1)|,|\mathcal{G}(I_2)|$ are strictly less than $|\mathcal{G}(I)|$, by~induction, it follows that $I_1$ is either a principal ideal, a~Veronese ideal or~a squarefree Veronese ideal, and~the same possibilities occur for $I_2$. We distinguish the various possibilities.
		
		\begin{itemize}[leftmargin=*,labelsep=-2mm]
			\item[] \textbf{Case 1.} Let $I_2$ be a principal ideal, then $\pd S/I_2=1$. Thus, $\pd S/I_1=2$.
		\end{itemize}
		
		Under this assumption, $I_1$ cannot be principal because $\pd S/I_1=2\ge1$.
		
		Assume that $I_1$ is a Veronese ideal in $m$ variables. Since all variables of $S$ divide some monomial of $\mathcal{G}(I)$ and $I_2\subset I_1$, then $I_1=\mm^d$ or $I_1=\nn^d$, where $\mm=(x_1,\dots,x_n)$ and $\nn=(x_1,\dots,x_{n-1})$. Thus, $m=n$ or $m=n-1$. Lemma \ref{Lemma:depth-m-Ind-VS} implies $m=\pd S/I_1=2$. So, $n=2$ or $n=3$.
		
		If $n=2$, then $I=x_2(x_1,x_2)^{d-1}+(x_1^d)=(x_1,x_2)^d$, which is CM and Veronese, or~\linebreak  $I=x_2(x_1^{d-1})+(x_1^d)=x_1^{d-1}(x_1,x_2)$, which is not~CM.
		
		Otherwise, if~$n=3$, then $I=x_3(x_1,x_2,x_3)^{d-1}+(u)$ or else $I=x_3(x_1,x_2)^{d-1}+(u)$ where $u\in K[x_1,x_2]$ is a monomial of degree $d$. If~$d=2$, then one easily sees that only in the second case and for $u=x_1x_2$ we have that $I$ is 
		a CM polymatroidal ideal, which is the squarefree Veronese $I_{3,2}$. Otherwise, suppose $d\ge3$. We may assume that $x_1^2$ divides $u$. In~the first case, $(I:x_1)=x_3(x_1,x_2,x_3)^{d-2}+(u/x_1)$ is not principal, nor Veronese, nor 
		squarefree Veronese. Thus, by~induction, $(I:x_1)$ is not a CM polymatroidal ideal, and~by Corollary \ref{Cor:ICM-CPI-VS}, we deduce that $I$ is also not CM. Similarly, in~the second case, we see that $(I:x_1)=x_3(x_1,x_2)^{d-1}+(u/x_1)$, and~thus also $I$, is not a CM polymatroidal~ideal.
		
		Assume now that $I_1$ is a squarefree Veronese ideal in $m$ variables. Then as argued in the case 1.2 we have $m=n$ or $m=n-1$. Lemma \ref{Lemma:depth-m-Ind-VS} gives $\pd S/I_1=m+1-(d-1)=2$. Thus, $d=m$. Hence $d=n$ or $d=n-1$. So $I=x_nI_{n,n-1}+(u)$ or $I=x_nI_{n-1,n-2}+(u)$ where $u\in K[x_1,\dots,x_{n-1}]$ is a monomial of degree $d=n$ in the first case or $d=n-1$ in the second case. In~the first case there is $i$ such that $x_i^2$ divides $u$. Say $i=1$. Then $(I:x_1)=x_n(I_{n,n-1}:x_1)+(u/x_1)$ is not a principal ideal, neither a Veronese ideal, neither a squarefree Veronese. Therefore, by~induction $(I:x_1)$ we see that is not a CM polymatroidal ideal, and~by Corollary \ref{Cor:ICM-CPI-VS} $I$ is also not a CM polymatroidal ideal. Similarly, in~the second case, if~$u=x_1\cdots x_{n-1}$, then $I=I_{n,n-1}$ is a CM squarefree Veronese ideal. Otherwise, $x_i^2$, say with $i=1$, divides $u$, and then, arguing as before, we see that $I$ is not a CM polymatroidal ideal.
		
		\begin{itemize}[leftmargin=*,labelsep=-2mm]
			\item[]\textbf{Case 2.} Let $I_2$ be a Veronese ideal in $m$ variables, then $\pd S/I_2=m\le n-1$.
		\end{itemize}
		
		Under this assumption, $I_1$ cannot be principal because $\pd S/I_1=m+1>1$.
		
		Assume that $I_1$ is a Veronese ideal in $\ell$ variables. Then $\ell=n$ or $\ell=n-1$. Lemma \ref{Lemma:depth-m-Ind-VS} implies that $\ell=\pd S/I_1=\pd S/I_2+1=m+1$. Thus, $\ell=n$ and $m=n-1$ or $\ell=n-1$ and $m=n-2$. In~the first case, $I=x_n\mm^{d-1}+(x_1,\dots,x_{n-1})^d=\mm^d$ is a CM Veronese ideal. In~the second case, up~to relabeling, we can write $I=x_n(x_1,\dots,x_{n-1})^{d-1}+(x_1,\dots,x_{n-2})^d$. However, this ideal is not polymatroidal. Otherwise, by~the exchange property {\color{black}(Lemma \ref{lem:exchange}}) applied to $u=x_nx_{n-1}^{d-1}$ and $v=x_{n-2}^d$, we should have $x_{n-2}x_{n-1}^{d-1}\in I$, which is not the~case.
		
		Assume now that $I_1$ is a squarefree Veronese ideal in $\ell$ variables.Then $\ell=n$ or $\ell=n-1$. Lemma \ref{Lemma:depth-m-Ind-VS} implies that $\pd S/I_1=\ell+1-(d-1)=\ell+2-d=\pd S/I_2+1=m+1$. Thus, either $m=n+1-d$ or $m=n-d$. Up~to relabeling, we have either \linebreak  $I=x_nI_{n,d-1}+(x_1,\dots,x_{n+1-d})^d$ or $I=x_nI_{n-1,d-1}+(x_1,\dots,x_{n-d})^d$. If~$d=2$, then these ideals become either $I=(x_1,\dots,x_n)^2$, which is a CM Veronese ideal, or~\linebreak  $I=x_n(x_1,\dots,x_{n-1})+(x_1,\dots,x_{n-2})^2$, which is not polymatroidal because the exchange property does not hold for $u=x_nx_{n-1}$ and $v=x_{n-2}^2$ since $x_{n-1}x_{n-2}\notin I$. If~$d\ge3$, then the above ideals are not polymatroidal. In~the first case, the exchange property does not hold for $u=x_{n+1-d}^d\in I_2$ and $v=(x_{n+2-d}\cdots x_{n})x_n\in x_n I_1$, otherwise $x_jx_{n+1-d}^{d-1}\in I$ for some $n+2-d\le j\le n$, which is not the case.
		
		\begin{itemize}[leftmargin=*,labelsep=-2mm]
			\item[]	\textbf{Case 3.} Let  $I_2$ be a squarefree Veronese in $m$ variables, $m\le n-1$. Then Lemma \ref{Lemma:depth-m-Ind-VS} implies $\pd S/I_2=m+1-d$ with $d\le m$. Hence, $\pd S/I_1=m+2-d$.
		\end{itemize}
		
		In such a case, the ideal $I_1$ cannot be principal because $\pd S/I_1=m+2-d>1$.
		
		Assume that $I_1$ is a Veronese ideal in $\ell$ variables. Then $\ell=n$ or $\ell=n-1$. Lemma~\ref{Lemma:depth-m-Ind-VS} implies that $\ell=\pd S/I_1=m+2-d$. Hence, either $d=m-n+2$ or $d=m-n+3$. Since $m\le n-1$, either $d\le 1$ or $d\le2$. Only the case $d=2$ is possible. If~$d=2$, then $\ell=m=n-1$ and we have $I=x_n(x_1,\dots,x_{n-1})+(x_1,\dots,x_{n-1})^2$. This ideal is not CM, otherwise it would be height-unmixed. Indeed, $(I:x_n)=(x_1,\dots,x_{n-1})$ and $(I:x_{n-1})=(x_1,\dots,x_n)$ are two associated primes of $I$ having different heights.
		\\
		Finally, assume that $I_1$ is a squarefree Veronese ideal in $\ell$ variables. Then $\ell=n$ or $\ell=n-1$. Lemma \ref{Lemma:depth-m-Ind-VS} implies that $\pd S/I_1=\ell+1-(d-1)=\ell+2-d=m+2-d$. Thus, $\ell=m$ and so either $\ell=m=n$ or $\ell=m=n-1$. The~first case is impossible because $m\le n-1$. In~the second case, we have $I=x_nI_{n-1,d-1}+I_{n-1,d}=I_{n,d}$ which is a CM squarefree Veronese ideal.
	\end{proof}
	
	\subsection{Bi-Cohen-Macaulay~Graphs} Let $I\subset S$ be a squarefree monomial ideal. Then $I$ may be seen as the Stanley--Reisner ideal of a unique simplicial complex on the vertex set $\{1,\dots,n\}$. Attached to $I$ is the Alexander dual $I^\vee$, which is again a squarefree monomial ideal. We say that $I$ is \textit{bi-Cohen--Macaulay} (bi-CM for~short) if both $I$ and $I^\vee$ are CM. By~the Eagon--Reiner criterion \cite[Theorem 8.1.9]{HH2011}, $I$ has a linear resolution if and only if $I^\vee$ is CM. Hence, $I$ is bi-CM if and only if it is CM with linear resolution.\smallskip
	
	Let $G$ be a finite simple graph on the vertex set $V(G)=\{1,\dots,n\}$ with edge set $E(G)$. The~\textit{edge ideal} $I(G)$ of $G$ is the squarefree monomial ideal of $S$ generated by the monomials $x_ix_j$ with $\{i,j\}\in E(G)$ \cite{RV}. The~Alexander dual of $I(G)$ is the squarefree monomial ideal of $S$ generated by the squarefree monomial  $x_{i_1}\cdots x_{i_t}$ such that $\{i_1,\ldots, i_t\}$ is a minimal vertex cover of $G$ \cite{RV}. Such an ideal is denoted by $J(G)$ and, since its definition, it is often called the \textit{cover ideal} of $G$.

	We say that $G$ is a bi-CM graph if $I(G)$ is bi-CM.\smallskip
	
	Let $G$ be a graph. The
	~\textit{open neighborhood} of $i\in V(G)$ is the set
	$$
	N_G(i)\ =\ \{j\in V(G)\ :\ \{i,j\}\in E(G)\}.
	$$
	
	A graph $G$ is called \textit{chordal} if it has no induced cycles of a length bigger than three. 
	A \textit{perfect elimination order} of $G$ is an ordering $v_1,\dots,v_n$ of its vertex set $V(G)$ such that $N_{G_i}(v_i)$ induces a complete subgraph on $G_i$, where $G_i$ is the induced subgraph of $G$ on the vertex set $\{i,i+1,\dots,n\}$. Hereafter, if~$1,2,\dots,n$ is a perfect elimination order of $G$, we highlight it by $x_1>x_2>\dots>x_n$.  For~a \emph{complete} graph $G$, we mean  a graph  satisfying the property that every set $\{i,j\}$ with $i,j\in V(G)$, $i\ne j$ is an edge of $G$. 
	
	\begin{Theorem}[\label{Thm:Dirac}\textup{\cite{Dirac61}}] A finite simple graph $G$ is chordal if and only if $G$ admits a perfect elimination order.
	\end{Theorem}
	
	The edge ideals with linear resolution were classified by Fr\"oberg~\cite{Froberg88}. Recall that the \textit{complementary graph} $G^c$ of $G$ is the graph with vertex set $V(G^c)=V(G)$ and where $\{i,j\}$ is an edge of $G^c$ if and only if $\{i,j\}\notin E(G)$. A~graph $G$ is called \textit{cochordal} if and only if $G^c$ is~chordal.
	
	\begin{Theorem}[\label{Thm:Froberg}\textup{\cite[Theorem 1]{Froberg88}}] Let $G$ be a finite simple graph. Then $I(G)$ has a linear resolution if and only if $G$ is cochordal.
	\end{Theorem}
	
	We quote the next fundamental result which was proved by Moradi and Khosh-Ahang~\cite[Theorem 3.6, Corollary 3.8]{MKA16}.
	\begin{Proposition}\label{Prop:MKA16}
		Let $G$ be a finite simple graph. Then $I(G)$ has linear resolution if and only if $I(G)$ is vertex splittable. Furthermore, if~$x_1>\dots>x_n$ is a perfect elimination order of $G^c$, then
		$$
		I(G)\ =\ x_1(x_j:j\in N_G(1))+I(G\setminus \{1\})
		$$
		is a vertex splitting of $I(G)$.
	\end{Proposition}
	
	Combining the above result with Theorem \ref{Thm:CM-VS}, we obtain the next characterization of the bi-CM graphs.
	\begin{Theorem}\label{Thm:BCM-VS}
		For a finite simple graph $G$, the~following conditions are~equivalent.
		\begin{enumerate}
			\item[\textup{(a)}] $G$ is a bi-CM graph.
			\item[\textup{(b)}] $G^c$ is a chordal graph with perfect elimination order $x_1>\dots>x_n$ and
			$$
			|N_G(i)\cap\{i,\dots,n\}|\ =\ |N_G(j)\cap\{j,\dots,n\}|+(j-i),
			$$
			for all $1\le i\le j\le n$ such that $|N_G(i)\cap\{i,\dots,n\}|,|N_G(j)\cap\{j,\dots,n\}|>0$.
		\end{enumerate}
		In particular, if~any of the equivalent conditions hold, then
		$$
		\pd S/I(G)\ =\ |N_G(i)\cap\{i,\dots,n\}|+(i-1),
		$$
		for any $1\le i\le n$ such that $|N_G(i)\cap\{i,\dots,n\}|>0$.
	\end{Theorem}
	\begin{proof}
		We proceed by induction on $n=|V(G)|$. By~Theorem \ref{Thm:Froberg}, $G$ must be cochordal. Fix $x_1>\dots>x_n$, which is 
		a perfect elimination order of $G^c$. By~Proposition \ref{Prop:MKA16}, $I(G)=x_1(x_j:j\in N_G(1))+I(G\setminus \{1\})$ is a vertex splitting. Applying Theorem \ref{Thm:CM-VS}, $I(G)$ is CM if and only if \linebreak  $J=(x_j:j\in N_G(1))$ and $I(G\setminus \{1\})$ are CM and $\pd S/I(G)=\pd S/J=\pd S/I(G\setminus \{1\})+1$. $J$ is CM because it is an ideal generated by variables and $\pd S/J=|N_G(1)|$. Notice that $x_2>\dots>x_n$ is a perfect elimination order of $(G\setminus \{1\})^c$.
		
		If $I(G\setminus \{1\})=0$, then $\pd S/I(G)=\pd S/J=\pd S/I(G\setminus \{1\})+1=1$ and $I(G)$ is principal, say $I(G)=(x_1x_2)$. In~this case, the thesis~holds.
		
		Suppose now that $I(G\setminus \{1\})\ne0$. Then, by~induction on $n$, $I(G\setminus \{1\})$ is CM if and only if
		\begin{equation}\label{eq:NG1}
			|N_{G\setminus \{1\}}(i)\cap\{i,\dots,n\}|\ =\ |N_{G\setminus \{1\}}(j)\cap\{j,\dots,n\}|+(j-i),
		\end{equation}
		for all $2\le i\le j\le n$ such that $|N_{G\setminus \{1\}}(i)\cap\{i,\dots,n\}|,|N_{G\setminus \{1\}}(j)\cap\{j,\dots,n\}|>0$ and moreover
		\begin{equation}\label{eq:NG2}
			\pd S/I(G\setminus \{1\})\ =\ |N_{G\setminus \{1\}}(i)\cap\{i,\dots,n\}|+(i-2),
		\end{equation}
		for any $2\le i\le n$ such that $|N_{G\setminus \{1\}}(i)\cap\{i,\dots,n\}|>0$.
		
		Notice that $N_G(i)\cap\{i,\dots,n\}=N_{G\setminus \{1\}}(i)\cap\{i,\dots,n\}$ for all $2\le i\le n$. Thus, by 
		combining (\ref{eq:NG1}) and (\ref{eq:NG2}) with the equality $\pd S/I(G)=\pd S/J=\pd S/I(G\setminus \{1\})+1$, we see that $I(G)$ is CM if and only if
		\begin{align*}
			|N_G(1)|\ &=\ |N_G(1)\cap\{1,\dots,n\}|\ =\ |N_{G\setminus \{1\}}(i)\cap\{i,\dots,n\}|+(i-2)+1\\
			&=\ |N_G(i)\cap\{i,\dots,n\}|+(i-1),
		\end{align*}
		for all $2\le i\le n$ such that $|N_{G\setminus\{1\}}(i)\cap\{i,\dots,n\}|>0$.
		
		Thus, we deduce that $|N_G(i)\cap\{i,\dots,n\}|=|N_G(j)\cap\{j,\dots,n\}|+(j-i)$ for all \linebreak  $1\le i\le j\le n$ such that $|N_G(i)\cap\{i,\dots,n\}|,|N_G(j)\cap\{j,\dots,n\}|>0$, as~desired. The~inductive proof is complete.
	\end{proof}
	
	Notice that in the above characterization, the~field $K$ plays no role. In~other words, the~bi-CM property of edge ideals does not depend on the field $K$. This also follows from the work of Herzog and Rahimi \cite[Corollary 1.2 (d)]{HR16}, where other classifications of the bi-CM graphs are~given.\\
	
	We end this paper with a couple of examples of a bi-CM and a non-bi-CM graph.
	\begin{Examples}
		\rm (a) 
		Consider the graph $G$ on five vertices and its complementary graph $G^c$ depicted below in Figure~\ref{fig1}.
		\begin{figure}[H]
			\begin{tikzpicture}[scale=0.9]
				\filldraw (0,0) circle (2pt) node[below]{4};
				\filldraw (2,0) circle (2pt) node[below]{5};
				\filldraw (3.5,0) circle (2pt) node[below]{3};
				\filldraw (1,1) circle (2pt) node[above]{2};
				\filldraw (1,-1) circle (2pt) node[below]{1};
				\draw[-] (0,0) -- (1,1) -- (2,0) -- (3.5,0);
				\draw[-] (0,0) -- (1,-1) -- (2,0);
				\draw[-] (1,1) -- (1,-1);
				\filldraw (8,0) circle (2pt) node[below]{3};
				\filldraw (9.5,0) circle (2pt) node[below]{4};
				\filldraw (11,0) circle (2pt) node[below]{5};
				\filldraw (7.3,1) circle (2pt) node[above]{1};
				\filldraw (8.7,1) circle (2pt) node[above]{2};
				\draw[-] (7.3,1) -- (8,0) -- (11,0);
				\draw[-] (8.7,1) -- (8,0);
				\filldraw (-1,1.5) node[above]{$G$};
				\filldraw (6,1.5) node[above]{$G^c$};
			\end{tikzpicture}
			\caption{\label{fig1}A bi-CM graph.}
		\end{figure}
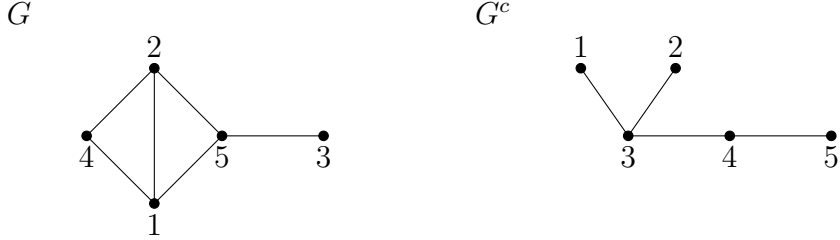
		Notice that $x_1>x_2>x_3>x_4>x_5$ is a perfect elimination order of $G^c$, so that $G^c$ is chordal (Theorem \ref{Thm:Dirac}). We have $|N_G(i)\cap\{i,\dots,5\}|>0$ only for $i=1,2,3$. It is easy to see that condition (b) of Theorem \ref{Thm:BCM-VS} is verified. Hence $G$ is bi-CM, as~one can also verify by using \textit{Macaulay2} \cite{GDS}.\\\\
		(b) Consider the graph $H$ and its complementary graph $H^c$ depicted below in Figure~\ref{fig2}.
		\begin{figure}[H]
			\begin{tikzpicture}[scale=0.9]
				\filldraw (0,0) circle (2pt) node[below]{4};
				\filldraw (2,0) circle (2pt) node[below]{5};
				\filldraw (3.5,0) circle (2pt) node[below]{3};
				\filldraw (1,1) circle (2pt) node[above]{2};
				\filldraw (1,-1) circle (2pt) node[below]{1};
				\draw[-] (0,0) -- (1,1) -- (2,0) -- (3.5,0);
				\draw[-] (0,0) -- (1,-1) -- (2,0);
				\filldraw (8,0) circle (2pt) node[below]{3};
				\filldraw (9.5,0) circle (2pt) node[below]{4};
				\filldraw (11,0) circle (2pt) node[below]{5};
				\filldraw (7.3,1) circle (2pt) node[above]{1};
				\filldraw (8.7,1) circle (2pt) node[above]{2};
				\draw[-] (7.3,1) -- (8,0) -- (11,0);
				\draw[-] (8.7,1) -- (8,0);
				\draw[-] (7.3,1) -- (8.7,1);
				\filldraw (-1,1.5) node[above]{$H$};
				\filldraw (6,1.5) node[above]{$H^c$};
			\end{tikzpicture}
			\caption{\label{fig2} A not bi-CM graph.}
		\end{figure}
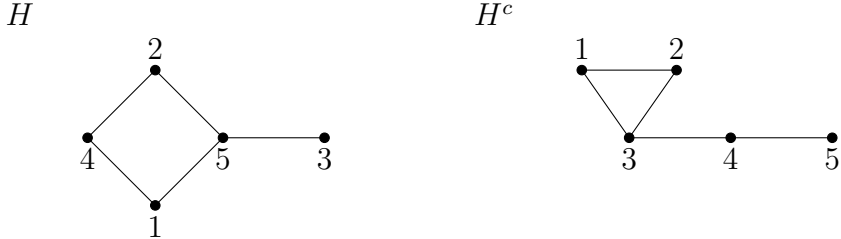
		As before, $x_1>x_2>x_3>x_4>x_5$ is a perfect elimination order of $H^c$, and~$H^c$ is chordal. We have $|N_H(i)\cap\{i,\dots,5\}|>0$ only for $i=1,2,3$. However, condition (b) of Theorem \ref{Thm:BCM-VS} is not verified. Indeed,
		$$
		|N_H(1)\cap\{1,\dots,5\}|=|N_H(2)\cap\{2,\dots,5\}|=|\{4,5\}|=2,
		$$
		but $|N_H(1)\cap\{1,\dots,5\}|\ne|N_H(2)\cap\{2,\dots,5\}|+1$. Hence, $H$ is not bi-CM. We can also verify this by using \textit{Macaulay2} \cite{GDS}. Indeed $J(H)$, the~cover ideal of $H$, is not~CM.
		
	\end{Examples}
	
	\section{Conclusions and~Perspectives} \label{sec:5}
	
	In view of our main Theorem \ref{Thm:CM-VS}, one can ask {\color{black} for} a similar criterion for the sequentially Cohen--Macaulayness of vertex splittable monomial ideals.
	
	\begin{Question}
		Let $I\subset S$ be a vertex splittable ideal, and~let $I=x_iI_1+I_2$ be a vertex splitting of $I$. Can we characterize the sequentially Cohen--Macaulayness of $I$ in terms of $I_1$ and $I_2$?
	\end{Question}
	
	This question could have interesting consequences for the theory of polymatroidal ideals. Indeed, a classification of the sequentially Cohen--Macaulay polymatroidal ideals has long been~elusive.
	
	On the computational side, to~check if a monomial ideal is vertex splittable is far {\color{black} easier} than to check if it admits a Betti splitting. Indeed, it is enough to check recursively Definition \ref{def:vertex splittable}. It could be useful to write a package in \textit{Macaulay2} \cite{GDS} that checks the vertex splittable property and some related properties.  
	
	\vspace{6pt} 
	

	\textit{Acknowledgments}: {We thank the anonymous referees for their careful reading and suggestions that improve the quality of the~paper.}\medskip

 \textit{Conflicts of interest}: {The authors declare no conflicts of~interest. 
}

\end{document}